\documentclass[]{interact}

\usepackage{epstopdf}
\usepackage[caption=false]{subfig}

\usepackage[numbers,sort&compress]{natbib}
\bibpunct[, ]{[}{]}{,}{n}{,}{,}

\theoremstyle{plain}
\newtheorem{theorem}{Theorem}[section]
\newtheorem{lemma}[theorem]{Lemma}

\theoremstyle{definition}

\theoremstyle{remark}

\begin{document}

\articletype{Research Article}

\title{Memory and Anticipation: Two main theorems for Markov regime-switching stochastic processes}

\author{
\name{E. Savku \textsuperscript{a}\textsuperscript{b}\thanks{CONTACT E. Savku. Email: esavku@gmail.com}}
\affil{\textsuperscript{a}University of Oslo, Department of Mathematics, Postboks 1053, Blindern, 0316, Oslo, Norway; \textsuperscript{b} Middle East Technical University, Institute of Applied Mathematics, Ankara, 06800, Turkey.}
}

\maketitle

\begin{abstract}
We present two main theorems for stochastic processes with a Markov regime-switching model. First, we work on an existence-uniqueness theorem for a Stochastic Differential Delay Equation with Jumps and Regimes (SDDEJRs). Then, we provide the duality between an SDDEJR and an Anticipated Backward Stochastic Differential Equation with Jumps and Regimes (ABSDEJRs). Our goal is to provide two technical and basis theorems for the future theoretical and applied developments of time-delayed and time-advanced models. 
\end{abstract}


\section{Introduction and Preliminaries}
\label{sec:intro}
\noindent In this research article, we purpose to provide generalized perspectives for \textit{time-delayed} and \textit{time-advanced} mathematical systems. We purpose to handle these two fundamental concept within the framework of a Markov regime-switching jump-diffusion model. The dynamics of such equations can be described as a  continuous-time process with discrete components. While the discrete variables belong to a countable set, the other dynamics evolve according to
a stochastic differential equation with jumps. Hence, sometimes such models are so-called Stochastic Hybrid Systems and can be considered as an interleaving between a finite or countable family of jump-diffusion processes. A detailed theory of regime-switching models can be found in \cite{my:15}.\\

\noindent In terms of real life applications, stochastic differential delay equations are expressed as models with \textit{memory}. Whereas this property brings a high level of technical challenge to the scientific projects, as a consequence of their power to demonstrate the worked problems in a more realistic way, such systems have been paid a significant attention by researchers. A comprehensive treatment of the theory of the stochastic differential delay equations (SDDEs) for diffusion processes can be found in the monograph by Mohammed \cite{m:10}.\\

\noindent On the other hand, in 2009, Peng and Yang, \cite{py:06}, introduced a new type of Backward Stochastic Differential Equations (BSDEs) called Anticipated (i.e., time-advanced) BSDEs and showed a duality relation between this type of equations and SDDEs in a diffusion setting. In this work, authors have also proved the existence-uniqueness conditions for such equations and also, demonstrated a comparison theorem. Later, some fundamental results for ABSDEs have been extended to jump-diffusion processes, see \cite{by:65} and \cite{th:05}. First attempts for proving the existence-uniqueness theorems and comparison theorems on Markov chains can be found in \cite{yang} and \cite{ymy:55}.  \\

\noindent Regime switching models are quite capable of catching abrupt changes in real life phenomenon as a consequence of  their heterogeneous nature. Hence, such systems with and without time-delayed and time-advanced systems have been studied for several years so far, especially in Stochastic Optimal Control. We can see a large literature on these concepts from narrower to most extended models in engineering,  economics, finance, cognitive sciences and neuroscience such as \cite{kieu, deepa, pamen, osz:07, savku3, sw:64, savku, zes:02} and references therein.\\
 \\
\noindent We demonstrate two main theorems for time-delayed and time-advanced systems of equations by a Markov regime-switching jump-diffusion model. Firstly, we prove an existence-uniqueness theorem for SDDEs with Jumps and Regimes (SDDEJRs) and then, we provide the duality relation between SDDEJRs and Anticipated BSDEs with jumps and regimes (ABSDEJRs). Yang, Mao and Yuan \cite{ymy:55} introduced an existence-uniqueness theorem for a \textit{delayed} stochastic differential diffusion process with regimes, but they did not prove it. Then, Bao and Yuan \cite{by:65} stated an extended form of this result for a jump-diffusion process without regimes and they did not prove it as well. To the best of our knowledge, our result is the first attempt, which  presents the technical conditions of such a theorem for a delayed regime-switching jump-diffusion process. \\

\noindent Peng and Yang \cite{py:06} constructed the duality between an SDDE and an ABSDE for a diffusion setting. They showed that the  solution of an ABSDE can be obtained by the solution of an SDDE. Later, Tu and Hao~\cite{th:05} extended it to a jump-diffusion process. In this paper, we extend this relation for the systems of Markovian regime-switching SDDEs and ABSDEs with jumps similar to Peng and Yang \cite{py:06}. \\

\noindent By extending and proving these two theorems of these fields, we purpose to open the doors for the construction of an extended new theory within the framework of Markov regime-switches. This work is organized as follows: In Section \ref{pre}, we present the dynamics, which appear as a consequence of Markov regime-switching architecture of the work. Section \ref{unique} provides the proof of existence-uniqueness theorem for such a large model. In Section \ref{duality}, we demonstrate the duality between SDDEJRs and ABSDEJRs.  Last section is devoted to a conclusion and outlook.  \\

\noindent Let us introduce the mathematical nature behind our work.\\

\section{Preliminaries}
\label{pre}

\noindent Throughout this work, we work with a finite time $T>0$, which represents the maturity time. Let $(\Omega,\mathbb{F},\left(\mathcal{F}_{t}\right)_{t\geq 0},\mathbb{P})$ be a complete probability space, where \\
$\mathbb{F}=\left(\mathcal{F}_{t}:t\in [0,T]\right)$ and $\left(\mathcal{F}_{t}\right)_{t\geq 0}$ is a right-continuous, $\mathbb{P}$-completed filtration generated by a Brownian motion $W(\cdot)$, a Poisson random measure $N(\cdot,\cdot)$ and a Markov chain $\alpha(\cdot)$. We assume that these processes are independent of each other and adapted to $\mathbb{F}$. Let $\mathcal{B}_{0}$ be the Borel $\sigma$-field generated by an open subset of $\mathbb{R}_{0}:=\mathbb{R}\setminus \left\{0\right\}$, whose closure does not contain the point $0$. Let $(N(dt,dz):t\in [0,T],z\in\mathbb{R}_{0})$
be the Poisson random measure on $([0,T]\times \mathbb{R}_{0},\mathcal{B}([0,T])\otimes \mathcal{B}_{0})$. We define:
\begin{equation*}
\tilde{N}(dt,dz):=N(dt,dz)-\nu(dz)dt
\end{equation*} 
as the compensated Poisson random measure, where $\nu$ is the L\'evy measure of the jump measure $N(\cdot,\cdot)$ such that
\begin{equation*} 
\nu(\left\{0\right\})=0 \quad  \texttt{and} \quad \int_{\mathbb{R}}(1\wedge |z|^{2})\nu(dz)<\infty. 
\end{equation*}
\noindent Furthermore,  let $\left(\alpha(t): t\in [0,T]\right)$ be a continuous-time, finite-state and observable Markov chain. The finite-state space of the homogenous and irreducible Markov chain $\alpha(t)$, $S=\left\{e_{1},e_{2},...,e_{D}\right\}$, is called a canonical state space, where $D\in \mathbb{N}$, $e_{i}\in \mathbb{R}^{D}$ and the $j$th component of $e_{i}$ is the Kronecker delta $\delta_{ij}$ for each pair of $i,j=1,2,...,D$. The generator of the chain under $\mathbb{P}$ is defined by $\Lambda:=[\lambda_{ij}]_{i,j=1,2,...,D}$. For each $i,j=1,2,...,D$, $\lambda_{ij}$ is the constant transition intensity of the chain from each state $e_{i}$ to state $e_{j}$ at time $t$. For $i\neq j$, $\lambda_{ij}\geq 0$ and $\sum_{j=1}^{D}\lambda_{ij}=0$; hence, $\lambda_{ii}\leq 0.$ We suppose that for each $i,j=1,2,...,D$, with $i\neq j$, $\lambda_{ij}>0$ and $\lambda_{ii}<0$.\\

\noindent We include regime-switches to our model as a compensated random measure generated by the Markov chain $\alpha$ as in \cite{zes:02}.\\
\noindent Let us give more details: \cite{eam:13} obtained the following semimartingale representation for a Markov chain $\alpha$:
\begin{equation*}
\alpha(t)=\alpha(0)+\int_{0}^{t}\Lambda^{T}\alpha(u)du+M(t),  
\end{equation*}
where $\left(M(t):t\in [0,T]\right)$ is an $\mathbb{R}^{D}$-valued, $(\mathbb{F},\mathbb{P})$-martingale and $\Lambda^{T}$ represents the transpose of the matrix. \\
Let $J^{ij}(t)$ represent the number of the jumps from the state $e_{i}$ to the state $e_{j}$ up to time $t$ for each $i,j=1,2,...,D$, with $i\neq j$ and $t\in [0,T]$. \\ Then,
\begin{align*}
J^{ij}(t)&:=\sum \limits_{0<s\leq t}\left\langle \alpha(s-),e_{i}\right\rangle \left\langle \alpha(s)-\alpha(s-),e_{j}\right\rangle\\
&=\int_{0}^{t}\left\langle \alpha(s-),e_{i}\right\rangle \left\langle d\alpha(s),e_{j}\right\rangle  \\
&=\int_{0}^{t}\left\langle \alpha(s-),e_{i}\right\rangle \left\langle \Lambda^{T}\alpha(s),e_{i}\right\rangle ds+\int_{0}^{t}\left\langle \alpha(s-),e_{i}\right\rangle \left\langle dM(s),e_{j}\right\rangle 
\end{align*}
\begin{align*}
&=\lambda_{ij}\int_{0}^{t}\left\langle \alpha(s-),e_{i}\right\rangle ds+m_{ij}(t),
\end{align*}
where the processes $m_{ij}$ are $(\mathbb{F},\mathbb{P})$-martingales and called the basic martingales associated with the chain $\alpha$. For each fixed $j=1,2,...,D$, let $\Phi_{j}$ be the number of the jumps into state $e_{j}$ up to time $t$. Then,
\begin{equation*}
\Phi_{j}(t):=\sum \limits_{i=1,i\neq j}^{D}J^{ij}(t)=\sum \limits_{i=1,i\neq j}^{D} \lambda_{ij}\int_{0}^{t}\left\langle \alpha(s-),e_{i}\right\rangle ds+ \tilde{\Phi}_{j}(t).
\end{equation*}
Let us define $\tilde{\Phi}_{j}(t):=\sum \limits_{i=1,i\neq j}^{D}m_{ij}(t)$ and $\lambda_{j}(t):=\sum \limits_{i=1,i\neq j}^{D}\lambda_{ij}\int_{0}^{t}\left\langle \alpha(s-),e_{i}\right\rangle ds$; then for each $j=1,2,...,D$,
\begin{equation*}
\tilde{\Phi}_{j}(t)=\Phi_{j}(t)-\lambda_{j}(t)
\end{equation*}
is an $(\mathbb{F},\mathbb{P})$-martingale. By $\tilde{\Phi}(t)=(\tilde{\Phi}_{1}(t),\tilde{\Phi}_{2}(t),...,\tilde{\Phi}_{D}(t))^{T}$, we represent a compensated random measure on $([0,T]\times S,\mathcal{B}([0,T])\otimes \mathcal{B}_{S})$, where $\mathcal{B}_{S}$ is a $\sigma$-field of $S$.\\

\noindent Please see Appendix \ref{ApA} for the descriptions of the spaces that we utilized along this paper.  Now, we can introduce our model and present our first theorem.

\section{Existence-Uniqueness Theorem for an SDDE with Jumps and Regimes}
\label{unique}

\noindent In this section, we prove the existence-uniqueness theorem for a system of Markov regime-switching jump-diffusion process with delay. The techniques applied here are inspired from \cite{khm:66}.\\

\noindent Let us represent our model:
\begin{align}
dX(t)=&\ b(t,X(t),X(t-\delta_{1}(t)),\alpha(t))dt  \nonumber \\
  &+\sigma(t,X(t),X(t-\delta_{2}(t)),\alpha(t))dW(t) \nonumber \\
	&+\int_{\mathbb{R}_{0}}\eta(t,X(t-),X((t-\delta_{3}(t))-),\alpha(t-),z)\tilde{N}(dt,dz) \nonumber\\
	&+\gamma(t,X(t-),X((t-\delta_{4}(t))-),\alpha(t-))d\tilde{\Phi}(t), \qquad t\in [0,T], \label{eq:4.1} \\		
X(t)=&\ x_{0}(t), \qquad t\in[-\delta,0] \label{eq:4.1.1}, 
\end{align}
where $x_{0}$ is a c\`{a}dl\`{a}g function defined from $[-\delta, 0]$ into $\mathbb{R}$ with the norm 
\begin{equation*}
\left\|x_{0}(t)\right\|=\sup_{-\delta\leq t\leq 0}|x_{0}(t)|.
\end{equation*} 
We may call $x_{0}$ as pre-history or initial path. The delay components, $\delta_{i}, \ i=1,2,3,4$, are nonnegative continuous real-valued functions defined on $[0,T]$ such that: 
\begin{description}
\item \textbf{(A1)} There exists a constant $\delta>0$ such that for each $t\in [0,T]$, 
\begin{equation*}
-\delta\leq t-\delta_{i}(t) \leq t, \qquad i=1,2,3,4.
\end{equation*}
\item \textbf{(A2)} There exists a constant $L>0$ such that for each $t\in [0,T]$ and for each non-negative, integrable $g(\cdot)$,
\begin{equation*}
\int_{0}^{t}g(s-\delta_{i}(s))ds\leq L\int_{-\delta}^{t}g(s)ds, \qquad i=1,2,3,4.
\end{equation*}
\end{description}
Let us consider the uniform case with $\delta_{i}(t)=\delta(t)$ for $i=1,2,3,4$, and assume $X(t-\delta(t))=Y(t), \ t\in [0,T]$. \\

\noindent Now we give further assumptions for the system (\ref{eq:4.1})-(\ref{eq:4.1.1}) without losing generality. \\

\noindent $b:[0,T]\times \mathbb{R}\times \mathbb{R}\times S\rightarrow \mathbb{R}$, \ $\sigma:[0,T]\times \mathbb{R}\times \mathbb{R}\times S\rightarrow \mathbb{R}$, \\
$\eta:[0,T]\times \mathbb{R}\times \mathbb{R}\times S \times \mathbb{R}_{0} \rightarrow \mathbb{R}$ and $\gamma:[0,T]\times \mathbb{R}\times \mathbb{R}\times S \rightarrow \mathbb{R}$ satisfy the following conditions: 
\begin{description}
\item \textbf{(H1)} There exists a constant $C>0$ such that for all $t\in [0,T]$, \ $e_{i}\in S$, \ $x_{1},\ x_{2},\ y_{1}$,\\ 
$y_{2}\in \mathbb{R}$,
\begin{align*}
&\left|b(t,x_{1},y_{1},e_{i})-b(t,x_{2},y_{2},e_{i})\right| \vee \left|\sigma(t,x_{1},y_{1},e_{i})-\sigma(t,x_{2},y_{2},e_{i})\right| \\
&\vee \left\|\eta(t,x_{1},y_{1},e_{i},z)-\eta(t,x_{2},y_{2},e_{i},z)\right\|_{J} \vee \left\|\gamma(t,x_{1},y_{1},e_{i})-\gamma(t,x_{2},y_{2},e_{i})\right\|_{S} \\
&\leq C(\left|x_{1}-x_{2}\right|+\left|y_{1}-y_{2}\right|).
\end{align*}
\item \textbf{(H2)} $b(\cdot,0,0,e_{i})\in L_{\mathbb{F}}^{2}(0,T;\mathbb{R})$, \ $\sigma(\cdot,0,0,e_{i})\in L_{\mathbb{F}}^{2}(0,T;\mathbb{R})$, \ $\eta(\cdot,0,0,e_{i},\cdot)$ \\ $\in \mathcal{H}^{2}_{\mathbb{F}}(0,T;\mathbb{R})$ \ and \ $\gamma(\cdot,0,0,e_{i})\in \mathcal{M}^{2}_{\mathbb{F}}(0,T;\mathbb{R}^{D})$ for all $e_{i}\in S$ and $t\in [0,T]$.
\end{description} 
\begin{theorem}
\label{th1}
 \noindent Under the assumptions (A1), (A2), (H1) and (H2), there exists a unique c\`{a}dl\`{a}g adapted solution $X(\cdot)\in L_{\mathbb{F}}^{2}(0,T;\mathbb{R})$ for Equations $(\ref{eq:4.1})-(\ref{eq:4.1.1})$.
\end{theorem}
\begin{proof} Let us fix $\beta=16C^{2}(1+L)+1$, where $C$ is the Lipschitz constant given in condition (H1) and $L$ is as in assumption (A2). Related to fixed $\beta$\ value, for the sake of convenience, we use a norm in the Banach space $L_{\mathbb{F}}^{2}(0,T;\mathbb{R})$ as 
\begin{equation*}
\left\|h(\cdot)\right\|_{\beta}^{2}=E\biggl[\int_{0}^{T}e^{-\beta s}\left\|h(s)\right\|^{2}ds\biggr],
\end{equation*}
which is equivalent to the original norm of $L_{\mathbb{F}}^{2}(0,T;\mathbb{R})$. \\

\noindent For any given $x(\cdot)\in L_{\mathbb{F}}^{2}(0,T;\mathbb{R})$ with $x(t)=x_{0}(t), \ t\in [-\delta, 0]$, we set:
\begin{align}
X(t)=&b(t,x(t),y(t),\alpha(t))dt+\sigma(t,x(t),y(t),\alpha(t))dW(t) \nonumber \\
&+\int_{\mathbb{R}_{0}}\eta(t,x(t-),y(t-),\alpha(t-),z)\tilde{N}(dt,dz) \nonumber \\
&+\gamma(t,x(t-),y(t-),\alpha(t-))d\tilde{\Phi}(t), \qquad t\in [0,T], \label{eq:small} \\
X(t)=&x_{0}(t), \qquad t\in [-\delta, 0] \label{eq:small1}.
\end{align}
Since $x(\cdot)$ is given, according to the existence-uniqueness results for the stochastic differential equations with jumps and regimes (see Proposition 7.1 by Cr\'epey~\cite{c:56}), the aforementioned system (\ref{eq:small})-(\ref{eq:small1}) has a unique solution in $L_{\mathbb{F}}^{2}(0,T;\mathbb{R})$. \\
Let us define a mapping as follows:
\begin{align*}
h:L_{\mathbb{F}}^{2}(0,T;\mathbb{R})&\rightarrow L_{\mathbb{F}}^{2}(0,T;\mathbb{R})\\
                            x(t)&\rightarrow h(x(t))=X(t), \quad t\in[-\delta,T].
\end{align*}
Note that $y(t)=x(t-\delta(t)), \ t\in [0,T]$ and $h$ is well-defined.\\

\noindent We will use the following abbreviations:
\begin{align*}
&b_{1}(s):=b(s,x_{1}(s),y_{1}(s),\alpha(s)),  \\
&b_{2}(s):=b(s,x_{2}(s),y_{2}(s),\alpha(s)), \ \hbox{etc.} 
\end{align*}
For arbitrary $x_{1}, \ x_{2}\in L_{\mathbb{F}}^{2}(0,T;\mathbb{R})$, we apply an extension of It\^o's formula (see Theorem 4.1 in \cite{zes:02} or Appendix \ref{ApB}) to $e^{-\beta t}|h(x_{1})(t)-h(x_{2})(t)|^{2}$ and take the expectation:
\begin{align*}
&E\biggl[e^{-\beta t}(h(x_{1})(t)-h(x_{2})(t))^{2}\biggr] \\
&=-\beta E\biggl[\int_{0}^{t}e^{-\beta s}(h(x_{1})(s)-h(x_{2})(s))^{2}ds\biggr]\\
&\qquad +2E\biggl[\int_{0}^{t}e^{-\beta s}(h(x_{1})(s)-h(x_{2})(s))\biggl\{(b_{1}(s)-b_{2}(s))ds \\
&\qquad+(\sigma_{1}(s)-\sigma_{2}(s))dW(s)+\int_{\mathbb{R}_{0}}(\eta_{1}(s,z)-\eta_{2}(s,z))\tilde{N}(ds,dz) \\
&\qquad+(\gamma_{1}(s)-\gamma_{2}(s))d\tilde{\Phi}(s)\biggr\}\biggr] \\
&\qquad+E\biggl[\int_{0}^{t}e^{-\beta s}(\sigma_{1}(s)-\sigma_{2}(s))^{2}ds\biggr] \\
&\qquad+E\biggl[\int_{0}^{t}\int_{\mathbb{R}_{0}}e^{-\beta s}(\eta_{1}(s,z)-\eta_{2}(s,z))^{2}\nu(dz)ds\biggr] \\
&\qquad+E\biggl[\int_{0}^{t}e^{-\beta s}\sum_{j=1}^D(\gamma_{1}^{j}(s)-\gamma_{2}^{j}(s))^{2}\lambda_{j}(s)ds\biggr].
\end{align*}
Since $a^{2}+b^{2}\geq 2ab$ and by assumption (A2), we get:
\begin{align*}
&E\biggl[e^{-\beta t}(h(x_{1})(t)-h(x_{2})(t))^{2}\biggr] \\
&=-\beta E\biggl[\int_{0}^{t}e^{-\beta s}(h(x_{1})(s)-h(x_{2})(s))^{2}ds\biggr] \\
&\qquad +E\biggl[\int_{0}^{t}e^{-\beta s}(h(x_{1})(s)-h(x_{2})(s))^{2}ds\biggr] \\
&\qquad +E\biggl[\int_{0}^{t}e^{-\beta s}|b_{1}(s)-b_{2}(s)|^{2}ds\biggr]+E\biggl[\int_{0}^{t}e^{-\beta s}|\sigma_{1}(s)-\sigma_{2}(s)|^{2}ds\biggr]  \\
&\qquad +E\biggl[\int_{0}^{t}e^{-\beta s}\left\|\eta_{1}(s)-\eta_{2}(s)\right\|_{J}^{2}ds\biggr]+E\biggl[\int_{0}^{t}e^{-\beta s}\left\|\gamma_{1}(s)-\gamma_{2}(s)\right\|_{S}^{2}ds\biggr] 
\end{align*}
\begin{align*}
&\leq (-\beta+1)E\biggl[\int_{0}^{t}e^{-\beta s}(h(x_{1})(s)-h(x_{2})(s))^{2}ds\biggr] \\
&\qquad +4C^{2}E\biggl[\int_{0}^{t}e^{-\beta s}(|x_{1}(s)-x_{2}(s)|+|y_{1}(s)-y_{2}(s)|)^{2}ds\biggr]\\
&\leq (-\beta+1)E\biggl[\int_{0}^{t}e^{-\beta s}(h(x_{1})(s)-h(x_{2})(s))^{2}ds\biggr] \\
& \qquad +8C^{2}E\biggl[\int_{0}^{t}e^{-\beta s}(|x_{1}(s)-x_{2}(s)|^{2}+|y_{1}(s)-y_{2}(s)|^{2}ds\biggr] \\
&\leq (-\beta+1)E\biggl[\int_{0}^{t}e^{-\beta s}(X_{1}(s)-X_{2}(s))^{2}ds\biggr] \\
&+8C^{2}E\biggl[\int_{0}^{t}e^{-\beta s}(|x_{1}(s)-x_{2}(s)|^{2}ds\biggr] \\
&+8LC^{2}E\biggl[\int_{-\delta}^{t}e^{-\beta s}(|x_{1}(s)-x_{2}(s)|^{2}ds\biggr].
\end{align*}
Note that for $s\in [-\delta,0], \ x_{1}(s)=x_{2}(s)=x_{0}(s)$; then we see:
\begin{align*}
&E\biggl[e^{-\beta t}(h(x_{1})(t)-h(x_{2})(t))^{2}\biggr]+(\beta-1)E\biggl[\int_{0}^{t}e^{-\beta s}(X_{1}(s)-X_{2}(s))^{2}ds\biggr] \\
&\leq  8C^{2}(1+L)E\biggl[\int_{0}^{t}e^{-\beta s}|x_{1}(s)-x_{2}(s)|^{2}ds\biggr].
\end{align*}
Let us also note that $E\biggl[e^{-\beta t}(h(x_{1})(t)-h(x_{2})(t))^{2}\biggr]>0$. \\

Since $\beta=16C^{2}(1+L)+1$, we obtain:
\begin{align*}
E\biggl[\int_{0}^{t}e^{-\beta s}|h(x_{1})(s)-h(x_{2})(s)|^{2}ds\biggr]\leq \frac{1}{2}E\biggl[\int_{0}^{t}e^{-\beta s}|x_{1}(s)-x_{2}(s)|^{2}ds\biggr].
\end{align*}
It is shown that $h$ is a contraction mapping in Banach space $L_{\mathbb{F}}^{2}(0,T;\mathbb{R})$. Hence, by Banach Fixed Point Theorem (see Appendix \ref{ApB}), there exists a unique solution \\
$X(\cdot)\in L_{\mathbb{F}}^{2}(0,T;\mathbb{R})$ \ for Equations (\ref{eq:4.1})-(\ref{eq:4.1.1}).
\end{proof}


\section{Duality Between SDDEs and ABSDEs with Jumps and Regimes}
\label{duality}

In this section, we provide the relation between an SDDEJRs and an ABSDEJRs. In the dynamics of this new type of BSDE, we clearly see the future values of the state processes as large as the delay component of an SDDE. A general representation of ABSDE with jumps and regimes can be seen in Apendix \ref{ApC}. The technique applied in this section to prove Theorem \ref{th1} is inspired from \cite{py:06}.\\

Let us state the duality theorem:
\begin{theorem}
\label{th2}
 Suppose $\delta>0$ is a given constant and $b,\bar{b}\in L_{\mathbb{F}}^{2}(t-\delta,T+\delta;\mathbb{R})$, $l\in L_{\mathbb{F}}^{2}(t,T;\mathbb{R})$, $\sigma,\bar{\sigma}\in L_{\mathbb{F}}^{2}(t-\delta,T+\delta;\mathbb{R})$, $\eta, \bar{\eta}\in H_{\mathbb{F}}^{2}(t-\delta,T+\delta;\mathbb{R})$, \\ $\gamma, \bar{\gamma}\in M_{\mathbb{F}}^{2}(t-\delta,T+\delta;\mathbb{R}^{D})$ and $b,\bar{b},\sigma,\bar{\sigma},\eta,\bar{\eta},\gamma,\bar{\gamma}$ are uniformly bounded. Then, for all $\xi \in S_{\mathbb{F}}^{2}(T,T+\delta;\mathbb{R}), \ \psi(t)\in L_{\mathbb{F}}^{2}(T,T+\delta;\mathbb{R}), \ \zeta\in H_{\mathbb{F}}^{2}(T,T+\delta;\mathbb{R})$ and \\
$\vartheta\in M_{\mathbb{F}}^{2}(T,T+\delta;\mathbb{R}^{D})$, the solution $Y$ of the following ABSDE,
\begin{align*}
-dY(s)=&\ \biggl(b(s,\alpha(s))Y(s)+\bar{b}(s,\alpha(s))E[Y(s+\delta)|\mathcal{F}_{s}] \\
&+\sigma(s,\alpha(s))Z(s)+\bar{\sigma}(s,\alpha(s))E[Z(s+\delta)|\mathcal{F}_{s}] \\
&+\int_{\mathbb{R}_{0}}Q(s,z)\eta(s,\alpha(s-),z)\nu(dz) \\
&+\int_{\mathbb{R}_{0}}E[Q(s+\delta,z)|\mathcal{F}_{s}]\bar{\eta}(s,\alpha(s-),z)\nu(dz) \\
&+\sum_{j=1}^{D}V^{j}(s)\gamma^{j}(s,\alpha(s-))\lambda_{j}(s) \\
&+\sum_{j=1}^{D}E[V^{j}(s+\delta)|\mathcal{F}_{s}]\bar{\gamma}^{j}(s,\alpha(s-))\lambda_{j}(s)+l(s,\alpha(s))\biggr)dt \\
&-Z(s)dW(s)-\int_{\mathbb{R}_{0}}Q(s,z)\tilde{N}(ds,dz)-V(s)d\tilde{\Phi}(s), \ s\in [t,T], 
\end{align*}
with terminal values, \ $Y(s)=\xi(s), \ Z(s)=\psi(s),\ Q(s)=\zeta(s)$ and $V(s)=\vartheta(s),$ \\
$s\in [T,T+\delta]$, can be given by the subsequent closed formula:
\begin{align*}
Y(t)=&\ E\biggl[X(T)\xi(T)+\int_{t}^{T}X(s)l(s,\alpha(s))ds \\
&+\int_{T}^{T+\delta}\biggl\{ \xi(s)\bar{b}(s-\delta,\alpha(s-\delta))X(s-\delta) \\
&+\psi(s)\bar{\sigma}(s-\delta,\alpha(s-\delta))X(s-\delta) \\
&+\int_{\mathbb{R}_{0}}\zeta(s,z)\bar{\eta}(s-\delta,\alpha((s-\delta)-),z)X((s-\delta)-)\nu(dz) \\
&+\sum_{j=1}^{D}\vartheta^{j}(s)\bar{\gamma}^{j}(s-\delta,\alpha((s-\delta)-))X((s-\delta)-)\lambda_{j}(s)\biggr\}ds| \mathcal{F}_{t} \biggr] 
\end{align*}
a.e., a.s., where $X(s)$ is the solution of the following SDDEJR with initial history:
\begin{align}
dX(s)=&\ \biggl(b(s,\alpha(s))X(s)+\bar{b}(s-\delta,\alpha(s-\delta))X(s-\delta)\biggr)ds \nonumber \\
&+\biggl(X(s)\sigma(s,\alpha(s))+X(s-\delta)\bar{\sigma}(s-\delta, \alpha(s-\delta))\biggr)dW(s) \nonumber \\
&+\int_{\mathbb{R}_{0}}\biggl(X(s-)\eta(s,\alpha(s-),z) \nonumber  \\
&+X((s-\delta)-)\bar{\eta}(s-\delta,\alpha((s-\delta)-),z)\biggr)\tilde{N}(ds,dz) \nonumber 
\end{align}
\begin{align}
&+\biggl(X(s-)\gamma(s,\alpha(s-))+X((s-\delta)-)\bar{\gamma}(s-\delta,\alpha(s-\delta)-)\biggr)d\tilde{\Phi}(s), \nonumber \\
\hbox{for} \ s\in &[t,T+\delta], \label{eq:sdded}\\
X(t)=&\ 1, \label{d1} \\
X(s)=&\ 0, \qquad s\in [t-\delta,t)  \label{d2}.
\end{align}
\end{theorem}
\begin{proof}
First, let us show that the system (\ref{eq:sdded})-(\ref{d1})-(\ref{d2}) has a unique solution.\\

When $s\in [t,t+\delta]$, the system becomes:
\begin{align}
\label{eq:nf}
dX(s)=& \ X(s-)\biggl\{b(s,\alpha(s))ds+\sigma(s,\alpha(s))dW(s)+\int_{\mathbb{R}_{0}}\eta(s,\alpha(s-),z)\tilde{N}(ds,dz) \nonumber \\
&+\gamma(s,\alpha(s-))\biggr\}d\tilde{\Phi}(s), \qquad s\in [t,t+\delta],  \\
X(t)=&\ 1. \nonumber
\end{align}
This is an SDE with jumps and regimes without delay and it is known that Equation (\ref{eq:nf}) has a unique solution in $L_{\mathbb{F}}^{2}(0,T;\mathbb{R})$ (cf. \cite{c:56}). \\
Let $\kappa(\cdot)$ be the solution of Equations (\ref{eq:nf}). For $s\in [t+\delta,T+\delta]$, Equation (\ref{eq:sdded}) becomes:
\begin{align}
\label{eq:sddejr}
dX(s)=&\ \biggl(b(s,\alpha(s))X(s)+\bar{b}(s-\delta,\alpha(s-\delta))X(s-\delta)\biggr)ds \nonumber \\
&+\biggl(X(s)\sigma(s,\alpha(s))+X(s-\delta)\bar{\sigma}(s-\delta, \alpha(s-\delta))\biggr)dW(s) \nonumber \\
&+\int_{\mathbb{R}_{0}}\biggl(X(s-)\eta(s,\alpha(s-),z) \nonumber\\
&+X((s-\delta)-)\bar{\eta}(s-\delta,\alpha((s-\delta)-),z)\biggr)\tilde{N}(dt,dz) \nonumber \\
&+\biggl(X(s-)\gamma(s,\alpha(s-))+X((s-\delta)-)\bar{\gamma}(s-\delta,\alpha((s-\delta)-))\biggr)d\tilde{\Phi}(t), \\
s\in & [t+\delta,T+\delta], \nonumber \\
X(s)=&\ \kappa(s), \qquad s\in [t,t+\delta] \nonumber .
\end{align}
This is a classical SDDE with jumps and regimes; hence, by Theorem (\ref{th1}), it is known that Equation (\ref{eq:sddejr}) has a unique solution. \\

If we apply the product rule to $X(s)Y(s)$ for $s\in [t,T]$ (cf. Lemma 3.2, by Zhang, Elliott and Siu \cite{zes:02} or Appendix \ref{ApB}), we obtain:
\begin{align*}
&X(T)Y(T)-X(t)Y(t) \\
&=-\int_{t}^{T} X(s-)\biggl\{ \biggl(b(s,\alpha(s))Y(s)+\bar{b}(s,\alpha(s))E[Y(s+\delta)|\mathcal{F}_{s}]\\
&\quad +\sigma(s,\alpha(s))Z(s)+\bar{\sigma}(s,\alpha(s))E[Z(s+\delta)|\mathcal{F}_{s}]
\end{align*}
\begin{align*}
&\quad +\int_{\mathbb{R}_{0}}(\eta(s,\alpha(s-),z)Q(s,z)+\bar{\eta}(s,\alpha(s-),z)E[Q(s+\delta,z)|\mathcal{F}_{s}])\nu(dz) \\
&\quad +\sum_{j=1}^{D}(\gamma^{j}(s,\alpha(s-))V^{j}(s)+\bar{\gamma}^{j}(s,\alpha(s-))E[V^{j}(s+\delta)|\mathcal{F}_{s}]))\lambda_{j}(t)\\
&\quad +l(s,\alpha(s))\biggr)ds-Z(s)dW(s)-\int_{\mathbb{R}_{0}}Q(s,z)\tilde{N}(ds,dz)-V(s)d\tilde{\Phi}(s)\biggr\}\\
&\quad +\int_{t}^{T}Y(s)\biggl\{ \biggl(b(s,\alpha(s))X(s)+\bar{b}(s-\delta,\alpha(s-\delta))X(s-\delta)\biggr)ds \\
&\quad +\biggl(\sigma(s,\alpha(s))X(s)+\bar{\sigma}(s-\delta,\alpha(s-\delta))X(s-\delta)\biggr)dW(s)\\
&\quad +\int_{\mathbb{R}_{0}}\biggl(\eta(s,\alpha(s-),z)X(s-)+\bar{\eta}(s-\delta,\alpha((s-\delta)-),z)X((s-\delta)-)\biggr)\tilde{N}(ds,dz) \\
&\quad +\biggl(\gamma(s,\alpha(s-))X(s-)+\bar{\gamma}(s-\delta,\alpha((s-\delta)-)X((s-\delta)-)\biggr)d\Phi(s)\biggr\} \\
&\quad +\int_{t}^{T}\biggl\{\biggl(\sigma(s,\alpha(s))X(s)+\bar{\sigma}(s-\delta)X(s-\delta)\biggr)Z(s)+\int_{\mathbb{R}_{0}}\biggl(\eta(s,\alpha(s-),z) \\
&\quad \times X(s-)+\bar{\eta}(s-\delta,\alpha((s-\delta)-),z)X((s-\delta)-)\biggr)Q(s,z)\nu(dz) \\
&\quad +\sum_{j=1}^{D}\biggl(\gamma^{j}(s,\alpha(s-))X(s-)+\bar{\gamma}^{j}(s-\delta,\alpha((s-\delta)-))\\
&\quad \times X((s-\delta)-)\biggr)V^{j}(s)\lambda_{j}(s)\biggr\}ds.
\end{align*}
Let us arrange the terms and take conditional expectation with respect to $\mathcal{F}_{t}$. Then, we get:
\begin{align*}
&E\left[X(T)Y(T)-X(t)Y(t)|\mathcal{F}_{t}\right] \\
&=E\biggl[\int_{t}^{T}\biggl\{ \bar{b}(s-\delta,\alpha(s-\delta))Y(s)X(s)-\bar{b}(s,\alpha(s))E[Y(s+\delta)|\mathcal{F}_{s}]X(s)\\
&\quad +\bar{\sigma}(s-\delta,\alpha(s-\delta))Z(s)X(s-\delta)-\bar{\sigma}(s,\alpha(s))E[Z(s+\delta)|\mathcal{F}_{s}]X(s) \\
&\quad -l(s,\alpha(s))X(s)-\int_{\mathbb{R}_{0}}\bar{\eta}(s,\alpha(s-),z)E[Q(s+\delta,z)|\mathcal{F}_{s}]X(s-)\nu(dz) \\
& \quad-\sum_{j=1}^{D}\bar{\gamma}^{j}(s,\alpha(s-))E[V^{j}(s+\delta)|\mathcal{F}_{s}]X(s-)\lambda_{j}(s)\\
& \quad +\int_{\mathbb{R}_{0}}\bar{\eta}(s-\delta,\alpha((s-\delta)-),z))Q(s,z)X((s-\delta)-)\nu(dz) \\
& \quad +\sum_{j=1}^{D}\bar{\gamma}^{j}(s-\delta,\alpha((s-\delta)-))V^{j}(s)X((s-\delta)-)\biggr\}ds|\mathcal{F}_{t} \biggr].
\end{align*}
We recall that $X(t)=1$ and $X(s)=0$ for $s\in [t-\delta,t)$. By tower property,  we obtain:
\begin{align*}
Y(t)=&\ E\biggl[X(T)Y(T)+\int_{t}^{T}X(s)l(s,\alpha(s))ds|\mathcal{F}_{t}\biggr] \\
&-E\biggl[\int_{t}^{T}X(s-\delta)Y(s)\bar{b}(s-\delta,\alpha(s-\delta))ds \\
&\qquad -\int_{t+\delta}^{T+\delta}X(s-\delta)Y(s)\bar{b}(s-\delta,\alpha(s-\delta))ds|\mathcal{F}_{t}\biggr] \\
&-E\biggl[\int_{t}^{T}X(s-\delta)Z(s)\bar{\sigma}(s-\delta,\alpha(s-\delta))ds \\
&\qquad -\int_{t+\delta}^{T+\delta}X(s-\delta)Z(s)\bar{\sigma}(s-\delta,\alpha(s-\delta))ds|\mathcal{F}_{t}\biggr] \\
&-E\biggl[\int_{t}^{T}\int_{\mathbb{R}_{0}}X((s-\delta)-)Q(s,z)\bar{\eta}(s-\delta,\alpha((s-\delta)-),z)\nu(dz)ds \\
&\qquad -\int_{t+\delta}^{T+\delta}\int_{\mathbb{R}_{0}}X((s-\delta)-)Q(s,z)\bar{\eta}(s-\delta,\alpha((s-\delta)-),z)\nu(dz)ds|\mathcal{F}_{t}\biggr] \\
&-E\biggl[\int_{t}^{T}\sum_{j=1}^{D}X((s-\delta)-)V^{j}(s)\bar{\gamma}^{j}(s-\delta,\alpha((s-\delta)-))\lambda_{j}(s)ds \\
&\qquad -\int_{t+\delta}^{T+\delta}\sum_{j=1}^{D}X((s-\delta)-)V^{j}(s)\bar{\gamma}(s-\delta,\alpha((s-\delta)-))\lambda_{j}(s)ds|\mathcal{F}_{t}\biggr]. \\
\end{align*}
Hence, we get a closed-form representation for $Y(t)$ as follows:
\begin{align*}
Y(t)=&\ E\biggl[X(T)\xi(T)+\int_{t}^{T}X(s)l(s,\alpha(s))ds \\
&+\int_{T}^{T+\delta}\biggl\{ \xi(s)\bar{b}(s-\delta,\alpha(s-\delta))X(s-\delta) \\
&+\psi(s)\bar{\sigma}(s-\delta,\alpha(s-\delta))X(s-\delta) \\
&+\int_{\mathbb{R}_{0}}\zeta(s,z)\bar{\eta}(s-\delta,\alpha((s-\delta)-),z)X((s-\delta)-)\nu(dz) \\
&+\sum_{j=1}^{D}\vartheta^{j}(s)\bar{\gamma}^{j}(s-\delta,\alpha((s-\delta)-))X((s-\delta)-)\lambda_{j}(s)\biggr\}ds|\mathcal{F}_{t}\biggr].
\end{align*}
\end{proof}
 
For an existence-uniqueness theorem of ABSDEs with jumps and regimes and an application of Theorem \ref{th1} and Theorem \ref{th2} to a stochastic optimal control problem, please see Savku and Weber \cite{sw:64}.


\section{Conclusion and Outlook}

This research paper aims to provide complementary results to provide the main basis for the development of ABSDEs with \textit{regimes}. These theorems may help for the related extensions in the theory of BSDEs such as in  reflected and doubly reflected BSDEs with \textit{regimes}. These advanced technical structures may lead many type of applications in stochastic optimal control, stochastic game theory, finance and insurance. For example, to the best of our knowledge, a comparison theorem for ABSDEJRs has not been proved yet, which becomes our next goal with its possible path into the game theory. Even some recent works \cite{fouque, han2} shows that time-delayed and time-advanced models evolves towards Deep Reinforcement Learning.  We believe that our technical results will serve as complementary theorems within the framework of memory and anticipation in a diversified field of study.

\section*{Disclosure statement and Funding}
There is no conflict of interest and this research did not have any receive any specific grant from funding agencies in the public, commercial, or not-for-profit sectors.

\bibliographystyle{tfs}
\bibliography{savkureferences}

\section{Appendices}

\appendix
\section{Complementary Remarks for Banach Spaces}
\label{ApA}

Firstly, let us introduce the following Banach spaces: \\

$L^{2}(\mathcal{F}_{T};\mathbb{R})=\{ \mathbb{R}$-valued, \ $\mathcal{F}_{T}$-measurable random variable $\phi$ such that
$E[\left|\phi \right|^{2}]<\infty\}$, \\ 

$L^{2}(\mathcal{B}_{0};\mathbb{R})=\{ \mathbb{R}$-valued, \ $\mathcal{B}_{0}$-measurable random variable $\phi$ such that \\
$\left\|\phi \right\|^{2}_{J}=\int_{\mathbb{R}_{0}}\left|\phi(z)\right|^{2}\nu(dz)<\infty \}$,\\

$L^{2}(\mathcal{B}_{S};\mathbb{R}^{D})=\{ \mathbb{R}^{D}$-valued, \ $\mathcal{B}_{S}$-measurable random variable $\phi$ such that \\
$\left\|\phi \right\|^{2}_{S}=\sum_{j=1}^{D}\left|\phi^{j} \right|^{2}\lambda_{j}(t)<\infty, \ j=1,2,...,D \}$,\\

$L^{2}(\mathcal{F}_{T}\times \mathcal{B}_{0};\mathbb{R})=\{ \mathbb{R}$-valued, $\mathcal{F}_{T}\times \mathcal{B}_{0}$-measurable random variable $\phi$ such that \\
$E[\int_{\mathbb{R}_{0}}\left|\phi(z) \right|^{2}\nu(dz)]<\infty \}$,\\

$L^{2}(\mathcal{F}_{T}\times \mathcal{B}_{S};\mathbb{R}^{D})=\{ \mathbb{R}^{D}$-valued,  $\mathcal{F}_{T}\times \mathcal{B}_{S}$-measurable random variable $\phi$ such that$ \\ 
E[\sum_{j=1}^{D}\left|\phi^{j} \right|^{2}\lambda_{j}(t)]<\infty, \ j=1,2,...,D \}$,\\

$L_{\mathbb{F}}^{2}(0,T;\mathbb{R})=\{ \mathbb{R}$-valued, \ $\mathcal{F}_{t}$-adapted stochastic process $\phi$ such that \\
$E[\int_{0}^{T}\left|\phi(t) \right|^{2}dt]<\infty \}$, \\ 

$S_{\mathbb{F}}^{2}(0,T;\mathbb{R})$=\{c\`{a}dl\`{a}g process $\phi$ in $L_{\mathbb{F}}^{2}(0,T;\mathbb{R})$ \ such that \ $E[\sup_{t\in [0,T]}\left|\phi(t) \right|^{2}]<\infty \}$, \\

$\mathcal{H}^{2}_{\mathbb{F}}(0,T;\mathbb{R})=\{ \mathbb{R}$-valued, \ $\mathcal{P}\otimes\mathcal{B}_{0}$-measurable stochastic process $\phi$ such that \\ 
$\left\|\phi(t) \right\|^{2}_{\mathcal{H}^{2}} =E[\int_{0}^{T}\left\|\phi(t) \right\|^{2}_{J}dt]<\infty \}$, \\

$\mathcal{M}^{2}_{\mathbb{F}}(0,T;\mathbb{R}^{D})=\{ \mathbb{R}^{D}$-valued, \ $\mathcal{P}\otimes \mathcal{B}_{S}-$measurable stochastic process $\phi$ such that \\
$\left\| \phi(t) \right\|^{2}_{\mathcal{M}^{2}}=E[\int_{0}^{T}\left\|\phi(t) \right\|^{2}_{S}dt]<\infty \}$.

\section{Complementary Remarks for Existence-Uniqueness Theorem}
\label{ApB}

Let us present a Markov regime-switching jump-diffusion model as follows:
\begin{align}
\label{eq:appd}
Y(t)=&\ b(t,Y(t),\alpha(t))dt+\sigma(t,Y(t),\alpha(t))dW(t) \nonumber \\
&+\int_{\mathbb{R}_{0}}\eta(t,Y(t-),\alpha(t-),z)\tilde{N}(dt,dz) \nonumber \\
&+\gamma(t,Y(t-),\alpha(t-))d\tilde{\Phi}(t), \qquad t\in [0,T],  \\
Y(0)=&\ y_{0}\in \mathbb{R}^{N}, \nonumber
\end{align}
where
\begin{align*}
&b:[0,T]\times \mathbb{R}^{N}\times S\rightarrow \mathbb{R}^{N}, \\
&\sigma:[0,T]\times \mathbb{R}^{N}\times S\rightarrow \mathbb{R}^{N\times M},\\
&\eta:[0,T]\times \mathbb{R}^{N}\times S\times \mathbb{R}_{0}\rightarrow \mathbb{R}^{N\times L}, \\
&\gamma:[0,T]\times \mathbb{R}^{N}\times S\times \rightarrow \mathbb{R}^{N\times D}
\end{align*}
are given functions. By Proposition 7.1 of Cr\'epey~\cite{c:56}, the system (\ref{eq:appd}) has a unique solution $Y(t)\in L_{\mathbb{F}}^{2}(0,T;\mathbb{R}^{N})$ under the following conditions:
\begin{description}
\item \textbf{(K1)} There exists a constant $K>0$ such that for all $t\in [0,T]$, \ $e_{i}\in S$, \ $x_{1},\ x_{2}\in \mathbb{R}^{N}$,
\begin{align*}
&\left\|b(t,x_{1},e_{i})-b(t,x_{2},e_{i})\right\|+\left\|\sigma(t,x_{1},e_{i})-\sigma(t,x_{2},e_{i})\right\| \\
&+\left\|\eta(t,x_{1},e_{i},z)-\eta(t,x_{2},e_{i},z)\right\|_{J}+\left\|\gamma(t,x_{1},e_{i})-\gamma(t,x_{2},e_{i})\right\|_{S} \\
&\leq K\left\|x_{1}-x_{2}\right\|.
\end{align*}
\item \textbf{(K2)} $b(\cdot,0,e_{i})\in L_{\mathbb{F}}^{2}(0,T;\mathbb{R}^{N})$, \ $\sigma(\cdot,0,e_{i})\in L_{\mathbb{F}}^{2}(0,T;\mathbb{R}^{N\times M})$, \ $\eta(\cdot,0,e_{i},\cdot)\in \mathcal{H}^{2}_{\mathbb{F}}(0,$\\$T;\mathbb{R}^{N\times L})$ \ and \ $\gamma(\cdot,0,e_{i})\in \mathcal{M}^{2}_{\mathbb{F}}(0,T;\mathbb{R}^{N\times D})$ for all $e_{i}\in S$ and $t\in [0,T]$.
\end{description}
Let us give the extension of It\^o's formula as in Zhang, Elliott and Siu~\cite{zes:02}.  
\begin{theorem}
\label{thm:ito}
Suppose an $N$-dimensional process $Y(t), \ t\in [0,T]$, is given as in System (\ref{eq:appd}) and the function $\phi(\cdot,\cdot,e_{i})\in C^{1,2}([0,T]\times \mathbb{R}^{N})$ \ for each \ $e_{j}\in S$. Then,
\begin{align*}
&\phi(T,Y(T),\alpha(T))-\phi(0,Y(0),\alpha(0)) \\
&=\int_{0}^{T}\biggl\{\biggl(\frac{\partial{\phi}}{\partial{t}}(t,Y(t-),\alpha(t-))+\sum_{k=1}^{N}\frac{\partial{\phi}}{\partial{y_{k}}}(t,Y(t-),\alpha(t-))b_{k}(t,Y(t-),\alpha(t-))\biggr) \nonumber \\
&\quad+\frac{1}{2}\sum_{k=1}^{N}\sum_{n=1}^{N}\int_{0}^{T}\frac{\partial^{2}{\phi}}{\partial{y_{k}}\partial{y_{n}}}(t,Y(t-),\alpha(t-))\sum_{l=1}^{M}\sigma_{kl}\sigma_{nl}(t,Y(t-),\alpha(t-))  \nonumber \\
&\quad+\sum_{m=1}^{L}\int_{0}^{T}\int_{\mathbb{R}_{0}}\biggl(\phi(t,Y(t-)+\eta^{m}(t,Y(t-),\alpha(t-),z),\alpha(t-))-\phi(t,Y(t-),\alpha(t-)) \nonumber \\
&\quad-\sum_{n=1}^{N}\frac{\partial{\phi}}{\partial{y_{n}}}(t,Y(t-),\alpha(t-))\eta_{nm}(t,Y(t-),\alpha(t-),z)\biggr)\nu_{m}(dz) \nonumber 
\end{align*}
\begin{align*}
&\quad+\sum_{j=1}^{D}\int_{\mathbb{R}_{0}}\biggl(\phi(t,Y(t-)+\gamma^{(j)}(t,Y(t-),\alpha(t-)),e_{j})-\phi(t,Y(t-),\alpha(t-)) \nonumber \\
&\quad-\sum_{n=1}^{N}\frac{\partial{\phi}}{\partial{y_{n}}}(t,Y(t-),\alpha(t-))\gamma_{nj}(t,Y(t-),\alpha(t-))\biggr)\lambda_{j}(t)\biggr\}dt \\
&\quad+\int_{0}^{T}\sum_{k=1}^{N}\frac{\partial{\phi}}{\partial{y}_{k}}(s,Y(s-),\alpha(s-))\sum_{n=1}^{M}\sigma_{kn}(s,Y(s-),\alpha(s-))dW(t) \nonumber \\
&\quad+\int_{0}^{T}\sum_{m=1}^{L}\int_{\mathbb{R}_{0}}\biggl(\phi(s,Y(s-)+\eta^{(m)}(s,Y(s-),\alpha(s-),z),\alpha(s-)) \nonumber \\
&\quad -\phi(s,Y(s-),\alpha(s-))\biggr)\tilde{N}(ds,dz) \nonumber \\
&\quad+\int_{0}^{T}\sum_{j=1}^{D}\biggl(\phi(s,Y(s-)+\gamma^{(j)}(s,Y(s-),\alpha(s-)),e_{j}) \nonumber \\
&\quad-\phi(s,Y(s-),\alpha(s-))\biggr)d\tilde{\Phi}_{j}(s),
\end{align*} 
where $\eta^{(m)}$ and $\gamma^{(j)}$ represents the $m$th and $j$th columns of the matrices $\eta$ and $\gamma$, respectively. \\
\end{theorem}
Furthermore, we present product rule for Markov regime-switching jump-diffusion models as in Zhang, Elliott and Siu~\cite{zes:02}.
\begin{lemma}
Suppose that $Y ^{j}(t), j = 1, 2$, are processes defined by the forward SDEs,
\begin{align}
\label{eq:product}
Y^{j}(t)=&\ b^{j}(t,Y(t),\alpha(t))dt+\sigma^{j}(t,Y(t),\alpha(t))dW(t) \nonumber\\
&+\int_{\mathbb{R}_{0}}\eta^{j}(t,Y(t-),\alpha(t-),z)\tilde{N}(dt,dz)  \nonumber\\
&+\gamma^{j}(t,Y(t-),\alpha(t-))d\tilde{\Phi}(t), \qquad t\in [0,T],  \\
Y^{j}(0)=&\ y^{j}\in \mathbb{R}^{N}, \quad j=1,2, \nonumber
\end{align}
where $b^{j}(t)\in \mathbb{R}^{N}$, $\sigma^{j}(t)\in \mathbb{R}^{N\times M}$, $\eta^{j}(t):=[\eta_{nl}^{j}(t)]\in \mathbb{R}^{N\times L}$ and $\gamma^{j}(t):=[\gamma_{nl}^{j}(t)]$\\$\in \mathbb{R}^{N\times D}$, $t\in [0,T]$, are predictable processes such that the integrals in \ref{eq:product} exist. Then,
\begin{align*}
&\left\langle Y^{1}(T),Y^{2}(T)\right\rangle  \\
&=\left\langle y^{1},y^{2}\right\rangle+\int_{0}^{T}\left\langle Y^{1}(t-),dY^{2}(t)\right\rangle+\int_{0}^{T}\left\langle Y^{2}(t-),dY^{1}(t)\right\rangle \\
&\quad+\int_{0}^{T}\left[(\sigma^{1}(t,\alpha(t)))^{T}\sigma^{2}(t,\alpha(t))\right]dt \\
&\quad+\int_{0}^{T}\sum_{l=1}^{L}\sum_{n=1}^{N}\eta_{nl}^{1}(t,\alpha(t-),z)\eta_{nl}^{2}(t,\alpha(t-),z)\nu^{l}(dz)dt
\end{align*}
\begin{align*}
&\quad+\int_{0}^{T}\sum_{l=1}^{D}\sum_{n=1}^{N}\gamma_{nl}^{1}(t,\alpha(t-))\gamma_{nl}^{2}(t,\alpha(t-))\lambda_{l}(t)dt.
\end{align*}
\end{lemma}

Let us state very well-known Banach Fixed Point Theorem as follows: 
\begin{theorem}
\label{thm:bfp}
Let $(X,d)$ be a complete metric space and $T:X\rightarrow X$ be a map such that 
\begin{equation*}
d(Tx,Tx')\leq cd(x,x')
\end{equation*}
for some $0\leq c<1$ and all $x, \ x'\in X$. Then $T$ has a unique fixed point $x^{*}$ in $X$, i.e., $T(x^{*})=x^{*}$.
\end{theorem}

\section{Complementary Remarks for Duality Theorem}
\label{ApC}

Let us introduce a generalized form of BSDEs with jumps and regimes as in \cite{sw:64}:
\begin{align*}
-dY(t)=&\ f(t,Y(t),Z(t),Q(t),V(t),Y(t+\delta_{1}(t)),Z(t+\delta_{2}(t)), \nonumber \\
 &Q(t+\delta_{3}(t)),V(t+\delta_{4}(t)),\alpha(t))ds-Z(t)dW(t) \nonumber\\
 &-\int_{\mathbb{R}_{0}}Q(t,z)\tilde{N}(dt,dz) -V(t)d\tilde{\Phi}(t), \qquad t\in [0,T],\\
Y(t)=&\ \xi(t), \ Z(t)=\psi(t),\ Q(t)=\zeta(t), \ V(t)=\vartheta(t), \ t\in [T,T+K]. \nonumber
\end{align*}
Let $\delta_{i}(\cdot)$, $i=1,2,3,4$, be an $\mathbb{R}^{+}$-valued continuous functions on $[0,T]$. \\

Please see Savku and Weber \cite{sw:64} for an  existence-uniqueness theorem.

\end{document}